\documentclass{amsart}
\usepackage[american]{babel}      
\usepackage[utf8]{inputenc}
\usepackage{amsmath,amssymb,amsthm,verbatim}
\usepackage{graphicx}
\usepackage{color}

\title[A general approximation approach  for the simultaneous treatment ...]{A general approximation approach  for the simultaneous treatment of integral and discrete operators}
\author[G. Vinti]{Gianluca Vinti}      
\address{Department of Mathematics and Computer Sciences - University of Perugia, 
         1, Via Vanvitelli\\ 06123 Perugia -Italy\\
Orcid Id: 0000-0002-9875-2790}
\email{gianluca.vinti@unipg.it}      
\author[L. Zampogni]{Luca Zampogni}
\address{Department of Mathematics and Computer Sciences - University of Perugia,   
         1, Via Vanvitelli\\ 06123 Perugia - Italy\\
Orcid Id: 0000-0002-6047-2646}
\email{luca.zampogni@unipg.it}       
\date{}

 \newtheorem{thm}{Theorem}[section]
  
 \newtheorem{cor}[thm]{Corollary}
 
 \newtheorem{prop}[thm]{Proposition}

\def\XXint#1#2#3{{\setbox0=\hbox{$#1{#2#3}{\int}$ }
\vcenter{\hbox{$#2#3$ }}\kern-.6\wd0}}

\DeclareMathOperator{\supp}{Supp}

\DeclareMathOperator{\sinc}{sinc}
 \def\R{\mathbb R}

\def\Z{\mathbb Z}

\def\T2{{\mathbb T}^2}

\def\T2{{\mathbb T}^2}
\def\T3{{\mathbb T}^3}

\newcommand{\ma}{>}
\newcommand{\ap}{``}
\newcommand{\mi}{<}

\newcommand{\s}{\hspace{4ex}}
\newcommand{\rest}{\upharpoonleft}

\newcommand{\ep}{\varepsilon}

\numberwithin{equation}{section}

\begin{document}
 \subjclass[2010]{41A35, 46E30, 47A58}  
\keywords{Approximation by operators, sampling Kantorovich operators, convolution operators,  Durrmeyer generalized sampling series, Orlicz spaces, modular convergence, locally compact Hausdorff topological groups}       
\begin{abstract}
 In this paper we give a unitary approach for the simultaneous study of the convergence of discrete  and integral operators described by  means of a  family of linear continuous functionals acting on functions defined on locally compact Hausdorff topological groups. The general family of  operators introduced and studied  includes very well-known operators in the literature. We give results of   uniform convergence, and modular convergence  in the general setting of Orlicz spaces.The latter result allow us to cover many other settings as the $L^p$-spaces, the interpolation spaces, the exponential spaces and many others.
%
\end{abstract}
\maketitle    
\section{Introduction}
The aim of  this paper is to give a unitary approach for the simultaneous study of discrete and integral operators. In order to do this, we need to introduce a general setting dealing with   a family of linear continuous functionals acting on suitable functional spaces, consisting on functions defined in locally compact Hausdorff topological groups. More precisely, we introduce the following linear integral operators:
  \begin{equation}\tag{I}\label{int1}T_wf(z)=\int_H\chi_w(z-h_w(t))\cdot L_{h_w(t)}fd\mu_H(t),\end{equation} where $H$ and $G$ are locally compact Hausdorff topological groups provided with their regular Haar measures $\mu_H$ and $\mu_G$ respectively, $f\in M(G)$ (the space of measurable functions on $G$) is such that the above integral defining $T_w$ is well defined, $h_w:H\rightarrow h_w(H)\subset G$ are homeomorphisms,    $(L_{h_w(t)})_{t\in H}$ is a suitable family of linear operators $L_{h_w(t)}:M(G)\rightarrow\mathbb R,$ and  $(\chi_w)_{w>0}$ is a family of kernel functionals $\chi_w:G\rightarrow\mathbb R.$

For the above family \eqref{int1} we establish a uniform approximation result for the aliasing error $T_wf-f$ (Theorem 4.2) and also a modular convergence result in the general setting of Orlicz spaces (Theorem 4.7). As pointed out in the paper, this latter result implies convergence in $L^p$-spaces and in other particular Orlicz spaces, as for example interpolation spaces, exponential spaces and others.
In order to obtain the general result of Theorem 4.7 we first test the modular convergence for continuous functions with compact support (Theorem 4.3), we establish a modular continuity of the family \eqref{int1} (Proposition 4.5), and, using a density result (Proposition 4.6), we are able to prove the main result.

It is worthwhile noting  that the operators \eqref{int1} contain various families of discrete and integral operators, very well-known in the literature. Among them, putting $G= \R, H=\Z$ and $h_w(k)=t_k/w$, where $(t_k)_{k}$ is a suitable sequence of real numbers,   there are the generalized sampling series \cite{Bu,BSpS,BuSt2,BMV,BBSV3}
$$T_wf(x)=\sum_{k\in\mathbb Z}\chi_w(x-t_k/w)f\left(\dfrac{t_k}{w}\right),$$ where 
$L_{t_k/w}f=f\left(\dfrac{t_k}{w}\right),$
 the Kantorovich sampling series \cite{BBSV2,VZ1,CV1,VZ2,CV2,VZ3}
$$T_wf(x)=\sum_{k\in\mathbb Z}\chi_w(x-t_k/w)\dfrac{w}{t_{k+1}-t_k}\int_{t_k/w}^{t_{k+1}/w}f(z)dz,$$
where $$L_{t_k/w}f=\dfrac{w}{t_{k+1}-t_k}\int_{t_k/w}^{t_{k+1}/w}f(z)dz,$$
and the  Durrmeyer generalized sampling series \cite{BM3}
$$T_wf(x)=\sum_{k\in\mathbb Z}\chi_w(x-t_k/w)\cdot w\int_\mathbb R\psi(wu-t_k)f(u)du,\s x\in\mathbb R,$$ 
where $$L_{t_k/w}f:=w\int_\mathbb R \psi(wu-t_k)f(u)du,$$ and $\psi$ is a suitable kernel function. 

Note that also perturbated versions of the above operators can be produced, as shown in Section 3.

Moreover, different  choices of the topological groups, the corresponding Haar measures and  the families $h_w$ and $L_{h_w(t)}$ produce classical convolution operators (see \cite{BN}) and the Mellin convolution operators (see \cite{AV4,AV5}), as shown in Section 5, together with their multidimensional versions.  Since each of the above operators plays an important role in the approximation theory and in the sampling and signal theory,  it seems very useful to have at disposal a unifying approach. 
Note that  the study of the generalized sampling operators has been started by P. L. Butzer and his school at Aachen (see\cite{Bu,BSpS,BuH,BuSt2,BL1,Sts,BL2,BSS}),
while the Kantorovich sampling series has been  introduced and studied in e.g., \cite{BBSV2,VZ1,VZ2}. Both are   important   from a mathematical point of view and  also in signal reconstruction: indeed, the  generalized sampling operators  represent extensions of the classical sampling theorem, and  their Kantorovich version has an important meaning in the applications in signal processing, since, replacing the value $f\left(\dfrac{t_k}{w}\right)$ with the average of $f$ in a small interval close to $\dfrac{t_k}{w}$ (i.e., with the value $\displaystyle{w\int_{\frac{t_k}{w}}^{\frac{t_{k+1}}{w}}f(u)du}),$
they reduce the \ap time-jitter'' errors.  For related references on the above operators, including connections with the classical sampling, see e.g.,  \cite{Whit,Sha,Je,DoSi,Hi,Hi1,BD,HiSt,Kot,BMV} and   \cite{Kan,G,BK,CV,V2,CV1,CV2,VV2,CCMV1,CCMV2,CV3,CV4,CV5}.
.

In Section 5, we will show that, for the above particular cases of operators, the assumptions of our approach are all satisfied.

 In Section 6, we will present some graphical representations showing the approximation of the considered operators to the function $f$.

Finally, we point out that the case of  Durrmeyer generalized sampling series  is important since it represents a case of operators for which continuous functions with compact support are approximated modularly and not with respect to the norm convergence in the Orlicz space (i.e., for which Corollary \ref{c1} cannot be established and assumption (L3) becomes not trivial). 


\section{Preliminaries}

We review some basic concepts and notations which will be used throughout all the paper.
Let $(H,+)$ be a locally compact Hausdroff topological group, with $\theta_H$ as neutral element. It is well-known (see \cite{Ma,PON}) that there exists a unique (up to a multiplicative constant) left (resp. right) translation invariant regular Haar measure $\mu_H$ (resp. $\nu_H$). In general, if $A$ is a Borel set of $H$, and if $-A=\{-a\;|\;a\in A\}$, then $\mu_H(-A)=k\nu_H(A),$ for some nonzero constant $k\in\mathbb R$. The measures $\mu_H$ and $\nu_H$ coincide if and only if $A$ in an \it unimodular group. \rm In this case one has $\mu_H(A)=\nu_H(A)=\mu_H(-A)$ for every Borel set $A\subset H$. Compact groups, abelian groups and discrete groups are well-known examples of unimodular groups.

Now, let $(G,+)$ be a locally compact abelian topological group. Let $\mathcal B$ be a local base of its neutral element $\theta_G$. Let $M(G)$ denote the set of measurable bounded functions $f:M\rightarrow \mathbb R$. By $C(G)$ (resp. $C_c(G)$) we denote the subsets of $M(G)$ consisting on uniformly continuous and bounded (resp. continuous with compact support) functions $f:G\rightarrow \mathbb R$. It is understood that $C(G)$ and $C_c(G)$ are equipped with the standard $||\cdot||_\infty$-norm.

Let $\varphi:\mathbb R_0^+\rightarrow\mathbb R_0^+$ be a continuous function: it is called a $\varphi-$function if:
\begin{enumerate}\item $\varphi(0)=0$ and $\varphi(x)>0$ for all $x>0$;
\item $\varphi$ is non-decreasing on $\mathbb R_0^+$;
\item $\displaystyle{\lim_{x\rightarrow+\infty}\varphi(x)=+\infty}$.\end{enumerate}
If a $\varphi-$function $\varphi$ is chosen, 
the functional $$I_G^\varphi:M(G)\rightarrow[0,+\infty]:f\mapsto\int_G\varphi(|f(z)|)d\mu_G(z)$$
is called \it modular functional \rm on $M(G)$. 

The Orlicz space $L^\varphi(G)$ is the subset of $M(G)$ consisting of those functions $f\in M(G)$ such that $$I_G^\varphi(\lambda f)<+\infty$$ for some $\lambda>0,$ equipped with the (strong) norm $$||f||_\varphi:=\inf\{\lambda>0\;|\;I_G^\varphi(\lambda f)<+\infty\}.$$ The norm $||\cdot||_\varphi$ is called \it Luxemburg norm. \rm  Anyway, the most natural notion of convergence in $L^\varphi(G)$, which is weaker than the norm convergence, is the \it modular convergence\rm: a sequence $(f_n)\subset L^\varphi(\mathbb R)$ converges modularly to a function $f\in L^\varphi(\mathbb R)$ if there exists a number $\lambda>0$ such that $$\lim_{n\rightarrow+\infty}I^\varphi_G[\lambda(f_n-f)]=0.$$ It is not difficult to see that for a sequence to converge strongly in $L^\varphi(G)$ it is necessary and sufficient that the above limit is true \it for every \rm $\lambda>0$. However, in some important cases the modular convergence and the Luxemburg norm convergence on $L^\varphi(G)$ are equivalent: this happens if and only if the so-called $\Delta_2-$condition is fulfilled, namely if there exists a number $M>0$ such that $$\dfrac{\varphi(2x)}{\varphi(x)}\leq M,\;\;\;\mbox{for every}\;x>0.$$ It can be shown that the $\Delta_2$-condition is satisfied if and only if  $f\in E^\varphi(G)= \{f\in L^\varphi(G)\;|\;I_G^\varphi(\lambda f)<+\infty$ for every $\lambda>0\}$. For references on Orlicz spaces, see e.g., \cite{KR,M,RR1,RR2,BMV}.

Orlicz spaces arise as a natural generalization of $L^p$ spaces, and in fact if $1\leq p<\infty$ and $\varphi(x)=x^p$, then the Orlicz space generated by $I_G^\varphi$ is exactly the Lebesgue space $L^p(G)$. The function $\varphi(x)=x^p$ satisfies the $\Delta_2$-condition, hence modular convergence is equivalent to the strong Luxemburg one, and they are in turn equivalent to the standard $L^p$-norm.
 However, there are many other examples of Orlicz spaces, which play an important role in many different situations, like PDEs and functional analysis. We mention the so-called exponential spaces (\cite{H}), where the $\varphi-$function generating the Orlicz space is given by  $\varphi_\alpha(x)=\exp(x^\alpha)-1,$ for a certain number $\alpha>0$. The exponential space $L^{\varphi_\alpha}(G)$ is an example in which the modular convergence is not equivalent to the Luxemburg one, and indeed $\varphi_\alpha$ does not satisfy the $\Delta_2$-condition. Another interesting example is furnished by the so-called interpolation spaces or Zygmund spaces  (\cite{Stn,BR,EK}). The generating $\varphi-$function is given by $\varphi_{\alpha,\beta}(x)=x^\alpha\ln^\beta(e+x)$, for fixed $\alpha\geq 1$ and $\beta>0$.  The function $\varphi_{\alpha,\beta}$ satisfies the $\Delta_2$-condition, hence in $L^{\varphi_{\alpha,\beta}}(G)$ the modular and the Luxemburg convergence are equivalent.

\section{Basic assumptions and examples}

Let $H$ and $G$ be locally compact Haursdorff topological groups with regular Haar measures $\mu_H$ and $\mu_G$ respectively. We require that $G$ is abelian (but, a priori, this assumption is not required on $H$).  Assume further that for every $w>0$ there exists a map $h_w:H\rightarrow G$ which restricts to a homeomorphism from $H$ to $h_w(H)$. Along this Section and Section 4 we will make assumptions on both the family $(L_{h_w(t)})_{t\in H}$ and on the kernel functions $(\chi_w)_{w>0}$. In Section 5 we will show several examples of operators for which these assumptions are all satisfied.

For every $w>0$, let $(L_{h_w(t)})_{t\in H}$ be a family of linear operators $L_{h_w(t)}:M(G)\rightarrow\mathbb R$ such that 

\begin{itemize}
\item[\bf(L1)] \it when restricted to $L^\infty(G)$, the operators $L_{h_w(t)}$ are \bf uniformly bounded\it, i.e., if $\tilde L_{h_w(t)}:=L_{h_w(t)}\rest_{L^\infty(G)}:L^\infty(G)\rightarrow\mathbb R,$ then $$||\tilde L_{h_w(t)}||:=\sup_{||f||_\infty\leq 1}|L_{h_w(t)}f|\leq \Upsilon<\infty,$$ for some positive constant $\Upsilon$ and for every $w>0$ and $t\in H$.
\item [\bf(L2)] \it If $f\in C(G)$, then the family $(L_{h_w(t)})_{t\in H}$ \bf preserves the continuity \it of $f$ in the sense that for every $\ep>0$, there exist $\overline w(\ep)>0$ and a set $B_\ep\in\mathcal B$ such that if $z-h_w(t)\in B_\ep$ and $w>\overline w$, then $|L_{h_w(t)}f-f(z)|<\ep.$

\end{itemize}

Further, let $(\chi_w)_{w>0}\subset L^1(G)$ be a family of kernel functionals $\chi_w:G\rightarrow\mathbb R$, which is uniformly bounded in $L^1(G)$ (i.e., $||\chi_w||_1<\Gamma$, for some constant $\Gamma>0$ and for every $w>0$).

 Let us introduce some notation which will be useful from now on. For $w>0$ and $z\in G$, if $A\subset G$ is a measurable set, we define $$A_{w,z}:=\{t\in H\;|\;z-h_w(t)\in A\},$$
$$A_w=\{t\in H\;|\;h_w(t)\in A\},\;\;\;\;\Upsilon_w(A)=\mu_H(A_w).$$

 The family $(\chi_w)_w$ is assumed to be chosen in such a way that:
\begin{itemize}
\item[$(\chi_1)$] the map $t\mapsto\chi_w(z-h_w(t))\in L^1(H)$ for every $z\in G$ and $w>0$;
\item[\bf($\chi_2$)] for every $w>0$ and $z\in G$, $$\int_H\chi_w(z-h_w(t))d\mu_H(t)=1;$$
\item[\bf($\chi_3$)] for every $w>0$, $$M:=\sup_{z\in G}\int_H|\chi_w(z-h_w(t))|d\mu_H(t)\mi\infty;$$
\item[\bf($\chi_4$)] if $w>0$, $z\in G$ and $B\in\mathcal B$, then $$\lim_{w\rightarrow+\infty}\int_{H\setminus B_{w,z}}|\chi_w(z-h_w(t))|d\mu_H(t)=0,$$ uniformly with respect to $z\in G$;
\item[\bf($\chi_5$)] for every $\ep>0$ and for every compact set $K\subset G$, there exist a symmetric compact set $C\subset G$ such that, for every $t\in K_w,$ we have $$\int_{G\setminus C}\Upsilon_w(K)|\chi_w(z-h_w(t))|d\mu_G(z)<\ep,$$ for sufficiently large $w>0$.
\end{itemize}
\vspace{.5cm}
For $w>0$, we study the operator $T_w:M(G)\rightarrow R,$ defined by \begin{equation}\label{O}T_wf(z)=\int_H\chi_w(z-h_w(t))\cdot L_{h_w(t)}fd\mu_H(t),\end{equation} where $f\in M(G)$ is such that the above integral defining $T_w$ is well defined.

\bigskip

There are a lot of  known examples of operators which can be expressed in the form \eqref{O}. We describe some (but certainly not all) kind of operators which are included in the general setting introduced by \eqref{O}.

\bf(1) \rm\s If $G=\mathbb R$, $H=\mathbb Z$ with the counting measure $d\mu_H(t)$, \eqref{O} becomes a series, namely $$\sum_{k\in\mathbb Z}\chi_w(z-h_w(k))\cdot L_{h_w(k)}f.$$ This type of series has been introduced  in \cite{VZ3}. According to the form of the operators $L_{h_w(k)}$, we obtain various types of the so-called \it sampling series. \rm

If, for example,  $h_w(k)=t_k/w\in\mathbb R$, and $(t_k)_k$ is an increasing sequence of real numbers such that $$\lim_{k\rightarrow\pm\infty}t_k=\pm\infty,\s \delta<t_{k+1}-t_k<\Delta\;\;(\delta,\Delta>0),$$ we can define $L_{h_w(k)}:M(\mathbb R)\rightarrow\mathbb R$ by $$L_{h_w(k)}f=L_{t_k/w}f=f\left(\dfrac{t_k}{w}\right).$$ The operator $T_wf$ translates into the \bf generalized sampling series, \rm  namely $$T_w^{(1)}f(x)=\sum_{k\in\mathbb Z}\chi_w(x-t_k/w)f\left(\dfrac{t_k}{w}\right).$$

Instead, we can define $L_{t_k/w}f:M(\mathbb R)\rightarrow\mathbb R$ as $$L_{t_k/w}f=\dfrac{w}{t_{k+1}-t_k}\int_{t_k/w}^{t_{k+1}/w}f(z)dz.$$ In this case the operators \eqref{O} give rise to the \bf Kantorovich sampling series, \rm 
$$T_w^{(2)}f(x)=\sum_{k\in\mathbb Z}\chi_w(x-t_k/w)\dfrac{w}{t_{k+1}-t_k}\int_{t_k/w}^{t_{k+1}/w}f(z)dz.$$ 

In $T_w^{(2)}f$ the sampling values $f(t_k/w)$ are replaced by averages, with the effect of reducing the so-called time-jitter errors and of regularizing the series, allowing to study the convergence in a wider sense (see \cite{BBSV2,VZ1,CV6,CV7}). 

A more general class of sampling series is given by the so-called {\bf  Durrmeyer generalized sampling series}. 
Let $\psi\in L^1(\mathbb R)$ be such that $$\int_\mathbb R\psi(u)du=1,$$ and let
 $$L_{h_w(k)}f=L_{t_k/w}f:=w\int_\mathbb R \psi(wu-t_k)f(u)du.$$   We now have the series $$(T^{(3)}_wf)(x)=\sum_{k\in\mathbb Z}\chi_w(x-t_k/w)\cdot L_{t_k/w}f=$$ $$=\sum_{k\in\mathbb Z}\chi_w(x-t_k/w)\cdot w\int_\mathbb R\psi(wu-t_k)f(u)du,\s x\in\mathbb R.$$

Perturbations can be added to the family of operators $L_{h_w(k)}$ to emulate the effect of the so-called time-jitter errors: so, for instance we can take $L_{h_w(k)}=f(t_k/w+j_k(w))$ where $\displaystyle{\lim_{w\rightarrow+\infty}j_k(w)=0}$, uniformly with respect to $k\in\mathbb Z$, to obtain a ''time-jitter'' generalized sampling series; see e.g. \cite{BL2,VZ2}.  We can also produce series analogous to $T^{(2)}_wf(x)$ and $T^{(3)}_wf(x)$ by adding suitable time-jitter perturbations. 

\bf (2)  \rm\s If $H=G=\mathbb R$, then we retrieve some interesting integral operators, together with their Kantorovich versions. For instance, we can take $h_w(t)=t,$ and  $L_{h_w(t)}f=L_tf=f(t)$, and \eqref{O} translates into a convolution operator, namely
$$T_w^{(4)}f(x)=\int_{\mathbb R}\chi_w(x-t)f(t)dt.$$ However, other types of integral operators can be generated: for example, if we take $h_w(t)=t$ and $$L_{h_w(t)}=\dfrac{w}{2}\int_{t-1/w}^{t+1/w}f(u)du,$$  we obtain  $$T_w^{(5)}f(x)=\int_{\mathbb R}\chi_w(x-t)\left(\dfrac{w}{2}\int_{t-1/w}^{t+1/w}f(u)du\right)dt,$$ which is a Kantorovich version of a convolution operator. Note that, whenever $f\in L^1_{loc}(\mathbb R)$,  the factor $\dfrac{w}{2}\displaystyle{\int_{t-1/w}^{t+1/w}f(u)du}$ converges, as $w\rightarrow+\infty$, to $f(t)$ for a.a. $t\in\mathbb R$ (by the Lebesgue-Besicovich Theorem). This fact testifies that, asymptotically, the operator $T_w^{(5)}f$ can be compared to the standard convolution operator.

As another example, let us now consider $H=G=\mathbb R^+$. In this case, the group operation is the product, and the only regular Haar measure on $\mathbb R^+$ (up to multiplicative constants) is the logarithmic measure $d\mu(t)=\dfrac{dt}{t}.$ We can then set $h_w(t)=t$ and $L_{h_w(t)}f=L_tf=f(t)$. Under these assumptions, we are left with $$T_w^{(6)}f(x)=\int_0^{\infty}\chi_w\left(\dfrac{x}{t}\right)f(t)\dfrac{dt}{t},$$ which is the Mellin convolution operator. 

Moreover, we can modify the operators $L_{h_w(t)}f$ by taking $$L_{h_w(t)}f=\dfrac{1}{2\ln(1+1/w)}\int_{t\frac{w}{w+1}}^{t\frac{w+1}{w}}f(u)\dfrac{du}{u},$$ and obtain a Kantorivich version  of the Mellin convolution, namely $$T_w^{(7)}f(x)=\int_0^\infty\chi_w\left(\dfrac{x}{t}\right)\left(\dfrac{1}{2\ln(1+1/w)}\int_{t\frac{w}{w+1}}^{t\frac{w+1}{w}}f(u)\dfrac{du}{u}\right)\dfrac{dt}{t}.$$ Note that, for large values of $w>0$, the factor $\dfrac{1}{2\ln(1+1/w)}$ can be replaced simply by $w/2$.

Finally, the setting of locally compact topological groups allows us to cover also the multidimensional version of the above operators.

\section{Approximation results}

In this section we state and prove the results concerning the convergence of \eqref{O}. Besides Theorem 4.2., which proves the convergence of \eqref{O} to $f$ uniformly, for $f\in C(G)$,   Theorem \ref{tfin} proves the convergence of \eqref{O} to $f$ in a subspace $\mathcal Y$ of an Orlicz space $L^\varphi(G)$, where, from now on, $\varphi$ is a convex $\varphi$-function. We will achieve this result by density, and by a number of preliminary (but important) facts which need to be proved. In the course of the discussion, we will also state two more assumptions on the family of operators $(L_{h_w(t)})_{t\in H},$ which will in turn both restrict the admissible families  $(L_{h_w(t)})_{t\in H}$ and define the subspace $\mathcal Y$ in which the main convergence result can be proved.

\begin{prop}\label{p1} Let $f\in L^\infty(G)$. Then the operator $T_w$ is well defined for every $w>0$ and in fact $$|T_wf(z)|\leq M\Upsilon ||f||_\infty.$$ \end{prop}
\begin{proof}
Let us take $f\in L^\infty(G)$. 
Then \begin{equation*}\begin{split} &|T_wf(z)|\leq \int_H|\chi_x(z-h_w(t))|\cdot|L_{h_w(t)}f|d\mu_H(t)\leq\\&\leq \Upsilon||f||_\infty\int_H|\chi_w(z-h_w(t))|d\mu_H(t)\leq M\Upsilon ||f||_\infty.\end{split}\end{equation*} 
\end{proof}

Now we prove the first result of convergence.
\begin{thm}\label{t1} Let $f\in C(G)$. Then $T_wf$ converges uniformly to $f$, i.e.,    $$\lim_{w\rightarrow+\infty}||T_wf-f||_\infty=0.$$ 
\end{thm}

\begin{proof}
Let $f\in C(G)$. We have $$|T_wf(z)-f(z)|\leq \int_H|\chi_w(z-h_w(t))|\cdot|L_{h_w(t)}f-f(z)|d\mu_H(t):=A.$$
Fix $\ep>0$, and let $B\in\mathcal B$ be such that if $z-h_w(t)\in B$ then $|L_{h_w(t)}f-f(z)|<\ep$, for sufficiently large $w>0$, by property (L2). Now,
\begin{equation*}\begin{split}A&=\int_{B_{w,z}}|\chi_w(z-h_w(t))|\cdot|L_{h_w(t)}f-f(z)|d\mu_H(t)+\\&+\int_{H\setminus B_{w,z}}|\chi_w(z-h_w(t))|\cdot|L_{h_w(t)}f-f(z)|d\mu_H(t):=A_1+A_2.\end{split}\end{equation*}
We have $$A_1=\int_{B_{w,z}}|\chi_w(z-h_w(t))|\cdot|L_{h_w(t)}f-f(z)|d\mu_H(t)\leq \ep M,$$ by using the property ($\chi3$).
Now, concerning $A_2$, we have \begin{equation*}\begin{split}A_2&=\int_{H\setminus B_{w,z}}|\chi_w(z-h_w(t))|\cdot|L_{h_w(t)}f-f(z)|d\mu_H(t)\leq \\&\leq (\Upsilon+1)||f||_\infty\int_{H\setminus B_{w,z}}|\chi_w(z-h_w(t))|d\mu_H(t).\end{split}\end{equation*}
The proof now follows since $\ep$ is arbitrarily chosen and by using property ($\chi4$).

\end{proof}

Several examples show that the assumptions (L1)--(L2) alone do not guarantee the convergence of  $T_wf$ in Orlicz spaces, even in the case when $f\in C_c(G)$. What we need is an assumption which describes the behavior of the family $(L_{h_w(t)})$ outside the support of a function $f\in C_c(G)$. We state this assumption as follows:
\begin{itemize}
\item[\bf(L3)] \it
 Let $f\in C_c(G)$ and let $\supp f=\tilde K\subset G$. We assume that there exist a compact set $K\supset \tilde K$ and a constant $\alpha>0$ such that for every $\ep>0$ there exists $\overline w>0$ such that $$\int_{H\setminus K_w}\varphi(\alpha|L_{h_w(t)}f|)d\mu_H(t)<\ep,$$ for every $w>\overline w$.  \end{itemize}
With the help of this property, we can prove the following theorem.
\begin{thm}\label{t2} Let $f\in C_c(G)$ and let (L3) be valid. Then for every $\lambda<\dfrac{\alpha}{M}$ we have \begin{equation}\label{ccc}\lim_{w\rightarrow+\infty}I^\varphi(T_wf-f)=0,\end{equation} that is, $T_wf$ converges modularly to $f$ in $L^\varphi(G)$. \end{thm}
\begin{proof} Theorem \ref{t1} tells us that for every $\lambda>0$ we have \begin{equation}\label{e1}\lim_{w\rightarrow+\infty} I^\varphi(\lambda||T_wf-f||_\infty)=0.\end{equation}
Next, let $\tilde K=\supp f$ and let $K\supset \tilde K$ and $\alpha>0$ such that (L3) holds. We apply   the Vitali convergence Theorem  to the family of functions  $(\varphi(T_wf(\cdot)))_w.$ For, in correspondence to $\ep>0$ and $K\subset G$ as above, there exists a symmetric compact set $C\subset G$ such that  ($\chi5$) holds, i.e., $$\int_{G\setminus C}\Upsilon_w(K)|\chi_w(z-h_w(t))|d\mu_G(z)<\ep$$ for every $t\in K_w$ and for every sufficiently large $w>0$.

Using the Jensen's inequality and the Fubini-Tonelli Theorem, we have  \begin{equation*}\begin{split}J&=\int_{G\setminus C}\varphi(\lambda|T_wf(z)|)d\mu_G(z)=\\&=\int_{G\setminus C}\varphi\left(\lambda\left|\int_H\chi_w(z-h_w(t))\cdot L_{h_w(t)}fd\mu_H(t)\right|\right)d\mu_G(z)\leq\\&\leq
\dfrac{1}{\Upsilon_w(K)M}\int_H\left(\varphi\left(\lambda M|L_{h_w(t)}f|\right)\int_{G\setminus C}|\chi_w(z-h_w(t)|\Upsilon_w(K)d\mu_G(z)\right)d\mu_H(t)\leq\\&\leq \dfrac{1}{\Upsilon_w(K)M}\left[\int_{K_w}\dots +\int_{H\setminus K_w}\dots\right]:=J_1+J_2.
\end{split}\end{equation*}
Now,
\begin{equation*}\begin{split}J_1&=\dfrac{1}{\Upsilon_w(K)M}\int_{K_w}\left(\varphi\left(\lambda M|L_{h_w(t)}f|\right)\int_{G\setminus C}|\chi_w(z-h_w(t)|\Upsilon_w(K)d\mu_G(z)\right)d\mu_H(t)\leq\\&\leq \dfrac{1}{\Upsilon_w(K)M}\cdot\Upsilon_w(K)\varphi(\lambda M\Upsilon||f||_\infty)\cdot\ep=\dfrac{1}{M}\cdot\varphi(\lambda M\Upsilon||f||_\infty)\cdot \ep.
\end{split}\end{equation*}

Concerning $J_2$, choose $\lambda>0$ such that $\lambda M<\alpha.$ We have, using (L3), $$J_2\leq \dfrac{1}{M}||\chi_w||_1\ep,$$ for sufficiently large $w>0$. In conclusion, for sufficiently large $w>0$, $$J\leq \left(\varphi(\lambda M\Upsilon||f||_\infty)+||\chi_w||_1\right)\dfrac{\ep}{M}<\left(\varphi(\lambda M\Upsilon||f||_\infty)+\Gamma\right)\dfrac{\ep}{M},$$ where $\Gamma>0$ is an upper bound for the family $(\chi_w)_{w>0}$ in $L^1(G)$. Moreover, for every measurable set $A\subset G$ with $\mu_G(A)<\infty,$ we have $$\int_A\varphi(\lambda|T_wf(z)|)d\mu_G(z)\leq \mu_G(A)\varphi(\lambda M\Upsilon||f||_\infty).$$ Hence, for every $\ep>0$ it suffices to take $\delta\leq \dfrac{\ep}{\varphi(\lambda M\Upsilon||f||_\infty)}$ to have $$\int_A\varphi(\lambda|T_wf(z)|)d\mu_G(z)<\ep$$ if $\mu_G(A)<\delta.$ This shows that the Vitali convergence Theorem can be applied to the functions $\varphi(T_wf(\cdot))$, and, using \eqref{e1}, the proof follows at once.

\end{proof}

The limit \eqref{ccc} is valid when $\lambda M<\alpha.$ If the family $\{L_{h_w(t)}\}$ satisfies $L_{h_w(t)}f=0$ whenever $t\notin K_w$, then the assumption (L3) is not necessary, and there is no reason for choosing a particular $\alpha$ in the proof of Theorem \ref{t2}. Hence the operators $T_wf$ converge to $f$ in the stronger Luxemburg norm. We state this result as follows:
\begin{cor}\label{c1} Let $f\in C_c(G)$. Let $\tilde K=\supp f$ and let $K\supset \tilde K$ a compact set in $G$ which contains $\tilde K$. Assume that, for sufficiently large $w>0$, $L_{h_w(t)}f=0$ for every $t\notin K_w$. Then $$\lim_{w\rightarrow+\infty}||T_wf-f||_\varphi=0.$$\end{cor}

We point out that if one wants to extend the convergence in Orlicz spaces to a family larger than $C_c(G)$, then one has technical difficulties to face. Let us show what happens: by estimating $I^\varphi(\lambda T_wf)$ ($\lambda >0$), we obtain
\begin{equation}\label{est}\begin{split} I^\varphi(\lambda T_wf)&\leq \dfrac{1}{M}\int_H\left[\varphi(M\lambda |L_{h_w(t)}f|)\int_G|\chi_w(z-h_w(t))|d\mu_G(z)\right]d\mu_H(t)\leq\\ &\leq \dfrac{||\chi_w||_1}{M}\int_H\varphi(M\lambda|L_{h_w(t)}f|)d\mu_H(t),\end{split}\end{equation} by using again the Jensen's inequality and the Fubini-Tonelli Theorem. It is now clear that if we want to achieve a desired result of convergence in Orlicz space for a function $f\in L^\varphi(G),$ the last inequality in \eqref{est} should be compared to $I^\varphi(\lambda\beta f)$ for some $\beta>0$, in order to have a modular continuity of the operators. In general, this is not possible, and in fact one does not even  know if the values $I^\varphi(\lambda T_wf)$ exist. Hence we need an additional assumption, which determines a restriction on  both the family $\{L_{h_w(t)}\}$ and the functions $f\in L^\varphi(G)$ for which the convergence can be actually obtained. The assumption reads as follows.

\begin{itemize}\item[\bf(L4)] \it Let $\varphi$ be a convex $\varphi$-function. There exists a subspace $\mathcal Y\subset L^\varphi(G)$ with $C_c^\infty(G)\subset \mathcal Y$ such that for every $f\in \mathcal Y$ and for every $\lambda>0$ there exist  constants $c=c(\lambda,f),\beta=\beta(\lambda,f)>0$ satisfying $$\limsup_{w\rightarrow+\infty}||\chi_w||_1\int_H\varphi(\lambda|L_{h_w(t)}f|)d\mu_H(t)\leq cI^\varphi(\lambda\beta f).$$\end{itemize}\rm 
Later, we will discuss some examples for which (L4) is satisfied.

The first consequence of the assumption (L4) is  the fact that $T_w$ maps $\mathcal Y$ into $L^\varphi(G)$: indeed we have the following \begin{prop}Let $\varphi$ be a convex $\varphi-$function. Let (L3) and (L4) be satisfied. Then for every $f\in \mathcal Y$ and $\lambda>0$ we have 
 $$I^\varphi(\lambda T_wf)\leq \dfrac{c}{M}I^\varphi(\lambda\beta Mf)$$ for sufficiently large $w>0$. In particular $T_w:\mathcal Y\rightarrow L^\varphi(G),$ for sufficiently large $w>0$.\end{prop}
We now need a density result (see, e.g., \cite{BM0}).
\begin{prop} The set $C^\infty_c(G)$ is dense in $L^\varphi(G)$ with respect to the modular convergence.\end{prop}
\begin{thm}\label{tfin} Let $f\in \mathcal Y$ and let (L1)--(L4) be valid. Then there exists $\lambda>0$ such that $$\lim_{w\rightarrow+\infty}I^\varphi(\lambda(T_wf-f))=0,$$i.e., $T_wf$ converges modularly to $f$ as $w\rightarrow+\infty$.\end{thm} 
\begin{proof} Take $f\in \mathcal Y$, and fix $\ep>0$. We can find a constant $\overline\lambda>0$ such that for every $\ep>0$ there exists $g\in C^\infty_c(G)$ with $$I^\varphi(\overline\lambda (f-g))<\ep.$$ Moreover $$I^\varphi(\lambda(T_wg-g))<\ep$$ for every $\lambda<\dfrac{\alpha}{M}$ and for sufficiently large $w>0$, by Theorem \ref{t2}. Choose $\lambda\leq \min\left\{\dfrac{\alpha}{3M},\dfrac{\overline\lambda}{3},\dfrac{\overline\lambda}{3\beta M}\right\}.$ Then, for sufficiently large $w>0$, 
\begin{equation*}\begin{split} I^\varphi(\lambda(T_wf-f))&\leq I^\varphi(3\lambda(T_wf-T_wg))+I^\varphi(3\lambda(T_wg-g))+I^\varphi(3\lambda(f-g))\leq \\&\leq \dfrac{c}{M}I^\varphi(\overline\lambda(f-g))+I^\varphi(3\lambda(T_wg-g))+I^\varphi(\overline\lambda(f-g))\leq \\&\leq\left(\dfrac{c}{M}+2\right)\cdot \ep.
\end{split}\end{equation*}
The proof follows easily since $\ep$ can be chosen  arbitrarily.

\end{proof}

\section{Examples and Applications}

In this section we apply the  previous results to some operators which can be generated by \eqref{O}. Some of the examples which we will show below have been introduced in Section 3.

We first test the validity of assumptions (L1)--(L4) for the series  $T^{(1)}_wf(x) $ and $T^{(2)}_wf(x)$, since they are very well-known in the theory of sampling series.

We start with the operator  $$T_w^{(1)}f(x)=\sum_{k\in\mathbb Z}\chi_w(x-t_k/w)f\left(\dfrac{t_k}{w}\right).$$ In this case, $H=\mathbb Z$, $G=\mathbb R$, $h_w(t_k)=\dfrac{t_k}{w}$ and $L_{t_k/w}f=f(t_k/w)$, where $(t_k)_{k\in\mathbb Z}$ is an incresing sequence of real numbers such that 

$$\lim_{k\rightarrow\pm\infty}t_k=\pm\infty,\s \delta<\Delta_k:=t_{k+1}-t_k<\Delta\;\;(\delta,\Delta>0).$$ 
 Clearly, the family $(L_{t_k/w})_{k\in\mathbb Z}$ satisfies (L1) and (L2), and in fact $||\tilde L_{t_k/w}||=1$ for every $k\in\mathbb Z$ and $w>0$. Assumption (L3) is satisfied as well, and in particular if $f\in C_c(\mathbb R)$ and $K=[-\gamma,\gamma]=Supp(f)$, then $L_{t_k/w}f\equiv 0$ whenever $t_k/w\notin K$. It follows that the stronger version of Theorem \ref{t2} (i.e., Corollary \ref{c1}) is valid. Concerning (L4), we first observe that \begin{equation}\label{ex1}\limsup_{w\rightarrow+\infty}||\chi_w||_1\sum_{k\in\mathbb Z}\varphi\left(\lambda|f(t_k/w)|\right)\leq \limsup_{w\rightarrow+\infty}\dfrac{w||\chi_w||_1}{\delta}\sum_{k\in\mathbb Z}\dfrac{\Delta_k}{w}\varphi(\lambda|f(t_k/w)|).\end{equation} Now, the sum in the right-hand side of \eqref{ex1} is a Riemann sum, and $$\limsup_{w\rightarrow+\infty}\sum_{k\in\mathbb Z}\dfrac{\Delta_k}{w}\varphi(\lambda|f(t_k/w)|)=I^\varphi(\lambda f)$$ whenever $f\in E^\varphi(\mathbb R)\cap BV^\varphi(\mathbb R)$ ($BV^\varphi(\mathbb R)$ is the set of those functions such that $\varphi(\lambda|f|)\in BV(\mathbb R)$ for every $\lambda>0$); see, e.g., \cite{HaSh,BMV}. It follows that the limit in \eqref{ex1} is finite and gives exactly (L4) (with $\beta=1$) when $f\in E^\varphi(\mathbb R)\cap BV^\varphi(\mathbb R)$ and $$\limsup_{w\rightarrow+\infty}w||\chi_w||_1<+\infty.$$ The finitess of the above limit can be easily achieved if we consider a single function $\chi\in L^1(\mathbb R)$ and define the family of kernels as $\chi_w(x)=\chi(wx)$, so that $\chi_w(x-t_k/w):=\chi(wx-t_k)$, which is a common situation for the generalized sampling series. In view of these considerations, we can state the following 
\begin{thm}\label{tex1}Let $f\in E^\varphi(\mathbb R)\cap BV^\varphi(\mathbb R)$ and let the family $(\chi_w)_{w>0}$ satisfy $(\chi_1)-(\chi_5)$ together with  the additional assumption  $$\limsup_{w\rightarrow+\infty}w||\chi_w||_1<+\infty. $$ Then there exists a number $\lambda>0$ such that $$\lim_{w\rightarrow+\infty}I^\varphi(\lambda(T_w^{(1)}f-f))=0.$$\end{thm}
Note that a sufficient condition for $f$ to belong to $E^\varphi(\mathbb R)\cap BV^\varphi(\mathbb R)$ is that $f\in R(\mathbb R)\cap BV^\varphi(\mathbb R)$, where $R(\mathbb R)$ is the space of absolutely Riemann-integrable functions on $\mathbb R$.

Now we consider the Kantorovich sampling series $$T_w^{(2)}f(x)=\sum_{k\in\mathbb Z}\chi_w(x-t_k/w)\dfrac{w}{t_{k+1}-t_k}\int_{t_k/w}^{t_{k+1}/w}f(z)dz,$$ where $(t_k)_{k\in\mathbb Z}$ is a sequence of real numbers which satisfies the properties listed above. In this case $h_w(k)=t_k/w$ and $$L_{t_k/w}f=\dfrac{w}{\Delta_k}\int_{t_k/w}^{t_{k+1}/w}f(z)dz,\s \Delta_k=t_{k+1}-t_k.$$ Assumption (L1) is easily satisfied since $$||L_{t_k/w}||=\sup_{||f||_\infty\leq 1}\left|\dfrac{w}{\Delta_k}\int_{t_k/w}^{t_{k+1}/w}f(z)dz\right|=1.$$ Now, for (L2), if $f\in C(\mathbb R)$, then in particular $f$ is continuous at an  arbitrary point $x\in\mathbb R$; now, for every $\ep>0$, let $\gamma>0$ be such that $|f(x)-f(z)|\leq \ep$ whenever $|x-z|\leq \gamma.$ It follows that $$|L_{t_k/w}f-f(x)|=\left|\dfrac{w}{\Delta_k}\int_{t_k/w}^{t_{k+1}/w}f(z)dz-f(x)\right|\leq\dfrac{w}{\Delta_k}\int_{t_k/w}^{t_{k+1}/w}|f(z)-f(x)|dz.$$
Let $\overline w=\dfrac{2\Delta}{\gamma}$ and  $B_\ep=B(0,\gamma/2).$ The interval $(t_k/w,t_{k+1}/w)$ has length less that $\gamma/2$ as soon as $w>\overline w$, hence if $w>\overline w$ and $z\in (t_k/w,t_{k+1}/w),$ then $|x-z|\leq \gamma,$ hence $|f(x)-f(z)|\leq \ep.$ This means that, for every $\ep>0$, $$|L_{t_k/w}f-f(x)|\leq \dfrac{w}{\Delta_k}\int_{t_k/w}^{t_{k+1}/w}|f(z)-f(x)|dz\leq \ep,$$ whenever $|x-t_k/w|\leq \gamma/2$ and $w>\overline w=\dfrac{2\Delta}{\gamma},$ i.e., assumption (L2) is satisfied.
To check (L3), let $Supp(f)=\tilde K=[\tilde \gamma,\tilde\gamma]$. Let $\gamma=\tilde \gamma+\Delta$ and $K=[-\gamma,\gamma].$ If $t_k/w\notin K$, then, for every $w>0$,  $$\int_{t_k/w}^{t_{k+1}/w}f(z)dz=0.$$ This implies that also in this case  the stronger Corollary \ref{c1} holds. Now, it remains to understand what (L4) means. Using the convexity of $\varphi$, we can write \begin{equation*}\begin{split}&||\chi_w||_1\sum_{k\in\mathbb Z}\varphi\left(\lambda\left|\dfrac{w}{\Delta_k}\int_{t_k/w}^{t_{k+1}/w}f(z)dz\right|\right)\leq \dfrac{w||\chi_w||_1}{\delta}\sum_{k\in\mathbb Z}\int_{t_k/w}^{t_{k+1}/w}\varphi(\lambda|f(z)|)dz=\\ &= \dfrac{w||\chi_w||_1}{\delta}\int_\mathbb R\varphi(\lambda|f(z)|)dz=\dfrac{w||\chi_w||_1}{\delta}I^\varphi(\lambda f).\end{split}\end{equation*}

The above relation implies that (L4) is satisfied by every function $f\in L^\varphi(\mathbb R)$ as soon as $$\limsup_{w\rightarrow+\infty}w||\chi_w||_1<\infty,$$ as in the previous example. We can now write the following
\begin{thm}\label{ex2} Let the family $(\chi_w)_{w>0}$ satisfy $(\chi_1)-(\chi_5)$ together with the assumption that $$\limsup_{w\rightarrow+\infty}w||\chi_w||_1<\infty.$$ Let $f\in L^\varphi(\mathbb R)$. Then there exists a number $\lambda>0$ such that  
$$\lim_{w\rightarrow+\infty}I^\varphi(\lambda(T_w^{(2)}f-f))=0.$$\end{thm}

We now move our attention to cases in which \eqref{O} is actually an integral operator. So, let $H=G=\mathbb R$. Apart from the convolution operator $T_w^{(4)}f$ obtained in Section 3, we focus our attention on a Kantorovich version of the convolution operator, namely $T_w^{(5)}f$ (see again Section 3). So, let $h_w:\mathbb R\rightarrow\mathbb R:t\mapsto t,$ and set $$L_{h_w(t)}f:=\dfrac{w}{2}\int_{t-1/w}^{t+1/w}f(u)du.$$ We have $$T_w^{(5)}f(x)=\int_\mathbb R\chi_w(x-t)\dfrac{w}{2}\int_{t-1/w}^{t+1/w}f(u)du.$$ It is fairly clear that (L1) is satisfied, and in fact $||L_{h_w(t)}||\equiv 1.$

Assumption (L2) is valid as well: indeed, arguing as above, let  $f$ be continuous at a point $x\in\mathbb R$, choose $\ep>0$, and let $\gamma>0$ be such that $|f(x)-f(u)|<\ep,$ whenever $|x-u|<\ep.$ Set $B_\ep=B(0,\gamma/2)$ and $\overline w=4/\gamma.$ Then if $|x-t|<\gamma/2$, $w>\overline w,$ and $u\in(t-1/w,t+1/w)$, we have $|x-u|<|x-t|+|t-u|<\dfrac{\gamma}{2}+\dfrac{2}{w}<\gamma,$ hence $|f(x)-f(u)|<\ep.$ It follows that $$|L_{h_w(t)}f-f(x)|\leq \dfrac{w}{2}\int_{t-1/w}^{t+1/w}|f(u)-f(x)|du\leq \ep.$$ This proves the validity of (L2). To check (L3), let $f\in C_c(\mathbb R)$ and assume that $Supp(f)=[-\tilde \gamma,\tilde \gamma]$ Choose $\gamma=\tilde \gamma+2$, and set $ K=[-\gamma,\gamma]$. Then, for sufficiently large $w>0$  and $t\notin K$, one has $$\int_{t-1/w}^{t+1/w}f(u)du=0.$$ This implies that (L3) is valid, and in particular that for every $\ep\ma 0$, if $Supp(f)=[\tilde\gamma,\tilde\gamma]$, $\gamma=\tilde\gamma+2$ and $K=[-\gamma,\gamma]$, then $$\int_{t\notin K}\varphi\left(\lambda\left|\dfrac{w}{2}\int_{t-1/w}^{t+1/w}f(u)du\right|\right)dt=0$$ for every $\lambda>0$ and $w>0$. In this case, Corollary \eqref{ccc} is valid, hence $T_w^{(4)}f$ converges to $f$ in the Luxemburg norm in $L^\varphi(\mathbb R)$. It remains to see what (L4) means. For, we have \begin{equation*}\begin{split}&||\chi_w||_1\int_\mathbb R\varphi\left(\lambda\left|\dfrac{w}{2}\int_{t-1/w}^{t+1/w}f(u)du\right|\right)dt\leq\\&\leq \dfrac{w||\chi_w||_1}{2}\int_0^{2/w}\left[\int_\mathbb R\varphi(\lambda|f(s+t-1/w)|ds\right]dt=\\&=\dfrac{w||\chi_w||_1}{2}\int_0^{2/w}I^\varphi(\lambda f)dt=||\chi_w||_1 I^\varphi(\lambda f).
\end{split}\end{equation*} by using the Jensen's inequality, the Fubini Theorem and some change of variables. This means that (L4) is satisfied for every $f\in L^\varphi(\mathbb R)$, and we note that, in this case, no further assumptions are required on the kernels, since  $||\chi_w||_1\leq\Gamma<+\infty$ for every $w>0$ by assumption. So we can write
\begin{thm}\label{ex3} Let the family $(\chi_w)_{w>0}$ satisfy $(\chi_1)-(\chi_5)$. Then, if $f\in L^\varphi(\mathbb R)$, there exists a number $\lambda>0$ such that $$\lim_{w\rightarrow+\infty}I^\varphi(\lambda(T_w^{(5)}f-f))=0.$$\end{thm}

We now  discuss operators $T_w^{(3)}f$ which provide an example for which the assumption (L3) is valid, but Corollary \ref{c1} is not, i.e., the family $\{L_{h_w(k)}\}$ does not satisfy $L_{h_w(t)}f=0$ when $t\notin K_w$ (and $f$ has compact support).  For simplicity, assume $G=\mathbb R$, $H=\mathbb Z$, $t=k$, $w=n$, and $h_n(t)=k/n.$  
Let $\psi\in L^1(\mathbb R)$ be such that $$\int_\mathbb R\psi(u)du=1,$$ and let
 $$L_{k/n}f:=n\int_\mathbb R \psi(nu-k)f(u)du.$$ Choose a function $\chi\in L^1(\mathbb R)$ as above, and define $\chi_n(x)=\chi(nx).$  We now have the series $$(T^{(3)}_nf)(x)=\sum_{k\in\mathbb Z}\chi(nx-k)\cdot L_{k/n}f=\sum_{k\in\mathbb Z}\chi(nx-k)\cdot n\int_\mathbb R\psi(nu-k)f(u)du,\s x\in\mathbb R.$$ This is the Durrmeyer generalized sampling series. 

It is easy to show that, if $f$ has compact support $K$ and if  $[-\gamma,\gamma]\supset K$, then in general $$L_{k/n}f=n\int_{-\gamma}^\gamma\psi(nu-k)f(u)du\neq 0.$$ This implies that Corollary \ref{c1} does not hold in general, and only the weaker Theorem \ref{t2} concerning modular convergence can be established. \\ Now, we check the properties $(L_i)$, $i=1,\dots,4.$

$(L1)$ is easily satisfied, and in fact $$|L_{k/n}f|\leq  n||f||_\infty\int_\mathbb R|\psi(nu-k)|du=||f||_\infty\cdot ||\psi||_1,$$ hence $$||L_{k/n}||_\infty=||\psi||_1,$$ for every $k\in\mathbb Z$ and $n\in\mathbb N$. 

Now, the conservation of continuity $(L2)$. Let us set $$K_n(x)=n\psi(nx).$$ Then $$||K_n||_1=||\psi||_1,\s n\in\mathbb N.$$ 
We can rewrite $$L_{k/n}f=\int_\mathbb RK_n(u-k/n)f(u)du.$$ This shows that $L_{k/n}$ is the restriction of the classical convolution $$(L_nf)(v)=\int_\mathbb R K_n(u-v)f(u)du,$$ to the values $v=k/n$. Now, if $f$ is uniformly continuous and bounded, it is well-known that (see \cite{BN}) $$\lim_{n\rightarrow\infty}||L_nf-f||_\infty=0.$$ It follows that for every $\ep>0$, there exists a number $\overline n\in\mathbb N$ such that for every $n>\overline n$ and $v\in\mathbb R$, one has $$|(L_nf)(v)-f(v)|\leq \ep/2.$$ Setting $v=k/n$ ($n>\overline n$) one obtains $$|L_{k/n}f-f(k/n)|\leq \ep/2.\s k\in\mathbb Z.$$
Now, fix $n>\overline n$. There exists $\delta>0$ such that if $|x-k/n|<\delta$, then $|f(x)-f(k/n)|<\ep/2.$ Hence, if $|x-k/n|<\delta$, we have $$|L_{k/n}f-f(x)|\leq |(L_nf)(k/n)-f(k/n)|+|f(k/n)-f(x)|<\ep.$$ This shows that $(L2)$ is valid, if $f$ is uniformly continuous and bounded.

Now, we check $(L3)$. Let $f$ have as support the compact set $K\subset\mathbb R$. Let $\gamma>0$ be such that $[-\gamma,\gamma]\supset K$.  It suffices to show that  for every $\ep>0$   there exists a compact set $[-M_n,M_n]\supset K$ such that  $$\sum_{|k|>M_n}\varphi(|L_{k/n}f|)=\sum_{|k|>M_n}\varphi\left(n\left|\int_\mathbb R \psi(nu-k)f(u)du\right|\right)<\ep$$ for all sufficiently large $n\in\mathbb N$. The summands in the above series can be estimated by $$\varphi\left(n\left|\int_\mathbb R \psi(nu-k)f(u)du\right|\right)\leq \dfrac{1}{||\psi||_1}\int_\mathbb R|\psi(t)|\varphi\left(||\psi||_1f\left(\dfrac{k+t}{n}\right)\right)dt.$$ Set $$I_n=\sum_{k\in\mathbb Z}\varphi\left(||\psi||_1f\left(\dfrac{k+t}{n}\right)\right).$$ Now, for fixed $n\in\mathbb N$, the series $I_n$ has only a finite number of non-zero summands, namely those which are given by the $k$'s such that $$-n\gamma-t<k<n\gamma-t,$$ hence they lie in an interval at most of length $2n\gamma.$ It follows that $$I_n=\sum_{k\in\mathbb Z}\varphi\left(||\psi||_1f\left(\dfrac{k+t}{n}\right)\right)\leq 2n\gamma\varphi(||\psi||_1\cdot||f||_\infty).$$ This shows that $I_n$ converges totally, and  $$\dfrac{1}{||\psi||_1}\int_\mathbb R|\psi(t)|\left[\sum_{k\in\mathbb Z}\varphi\left(||\psi||_1f\left(\dfrac{k+t}{n}\right)\right)\right]dt\leq 2n\gamma\varphi(||\psi_1||\cdot||f||_\infty).$$ This implies that for every $\ep>0$ and $n\in\mathbb N$ there exists a number $M_n\in\mathbb N$ such that $$\sum_{|k|>M_n}\varphi\left(|L_{k/n}f|\right)<\ep.$$

It remains to prove $(L4)$. As before, set $K_n(x)=n\psi(nx).$  Then $$L_{k/n}f=\int_\mathbb R K_n(u-k/n)f(u)du.$$ We have $$||\chi_n||_1\cdot\sum_{k\in\mathbb Z}\varphi(\lambda|L_{k/n}f|)\leq n||\chi_n||_1\cdot\sum_{k\in\mathbb Z}\dfrac{1}{n}\varphi\left(\lambda\left|\int_\mathbb R K_n(u-k/n)f(u)du\right|\right).$$ The quantity $n||\chi_n||_1$ equals $||\chi||_1$, and the right-hand side of the above inequality is a Riemann sum of the function $\varphi\left(\lambda\left|L_nf(v)\right|\right)$ for $v=k/n$,  where, as before, $$(L_nf)(v)=\int_\mathbb R K_n(u-v)f(u)du.$$ If $f\in E^\varphi(\mathbb R)\cap BV^\varphi(\mathbb R)$ ($BV^\varphi(\mathbb R)$ is the set of those functions such that $\varphi(\lambda|f|)\in BV(\mathbb R)$ for every $\lambda>0$), then, since $L_nf$ is a convolution, $L_nf\in E^\varphi(\mathbb R)\cap BV^\varphi(\mathbb R)$ as well, hence $$\limsup_{n\rightarrow\infty}\sum_{k\in\mathbb Z}\dfrac{1}{n}\varphi\left(\lambda\left|\int_\mathbb R K_n(u-k/n)f(u)du\right|\right)=I^\varphi(\lambda L_nf).$$ It remains to show that there exist numbers $\beta,c>0$ such that $$I^\varphi(\lambda L_nf)\leq c I^\varphi(\lambda\beta f).$$ 
For, $$I^\varphi(\lambda L_nf)=\int_\mathbb R\varphi\left(\left|\lambda\int_\mathbb RK_n(u-v)f(u)du\right|\right)dv\leq$$ $$\leq \dfrac{1}{||K_n||_1}\int_\mathbb R\int_\mathbb R\varphi\left(\lambda||K_n||_1\left|f(u)\right|\right)du\cdot |K_n(u-v)|dv\leq$$ $$\leq \int_\mathbb R\varphi\left(\lambda||K_n||_1\left|f(u)\right|\right)du=\int_\mathbb R\varphi\left(\lambda||\psi||_1\left|f(u)\right|\right)du= I^\varphi(\lambda||\psi||_1f).$$

We have proved $(L4)$ (with $c=1$ and $\beta=||\psi||_1$). 

\section{Some Graphical Examples}
This section provides some graphical representations of the convergence of the operators we have studied in the previous sections. In all the examples below the convergence must be interpreted as to be in the $L^p$ setting.

Although the prototypical example of kernel is obtained from the Fejer's kernel function

\vspace{.5cm}
$$F(x)=\dfrac{1}{2}\sinc^2\left(\dfrac{x}{2}\right),$$ 
\vspace{.5cm}
where
\vspace{.5cm}
 $$\sinc(x)=\begin{cases}\dfrac{\sin\pi x}{\pi x},&x\in\mathbb R\setminus\{0\},\\1,&x=0\end{cases},$$
\vspace{.5cm}

 it will be convenient for computational purposes to take a kernel with compact support over $\mathbb R$. Well-known examples of such kernels are those arising from linear combinations the so-called $B$-splines functions of order $n\in\mathbb N$, namely

$$M_n(x)=\dfrac{1}{(n-1)!}\sum_{j=0}^n(-1)^j\begin{pmatrix}n\\j\end{pmatrix}\left(\dfrac{n}{2}+x-1\right)_+^{n-1},$$ 
where the symbol $(\cdot)_+$ denotes the positive part. Below we represent the graphs of the functions $M_3(x),M_4(x)$  and $M(x)=4M_3(x)-3M_4(x)$ (Figure 1). It is easy to see that $M(x)$ satisfies all the assumptions required in Section 3.


\begin{figure}[htbp]
\begin{center}\includegraphics[width=.45\textwidth,height=.34\textwidth]
{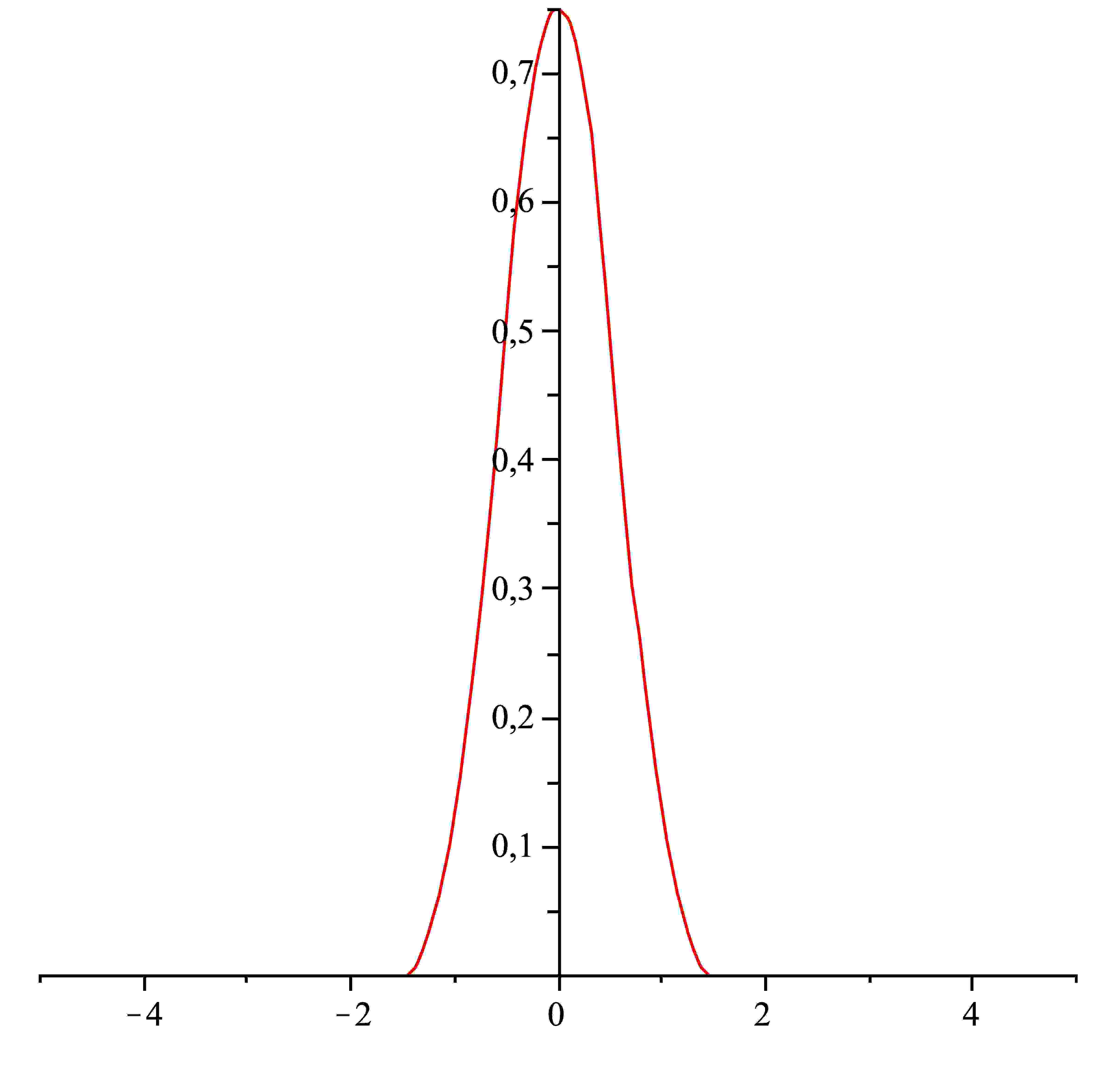}\includegraphics[width=.45\textwidth,height=.34\textwidth]
{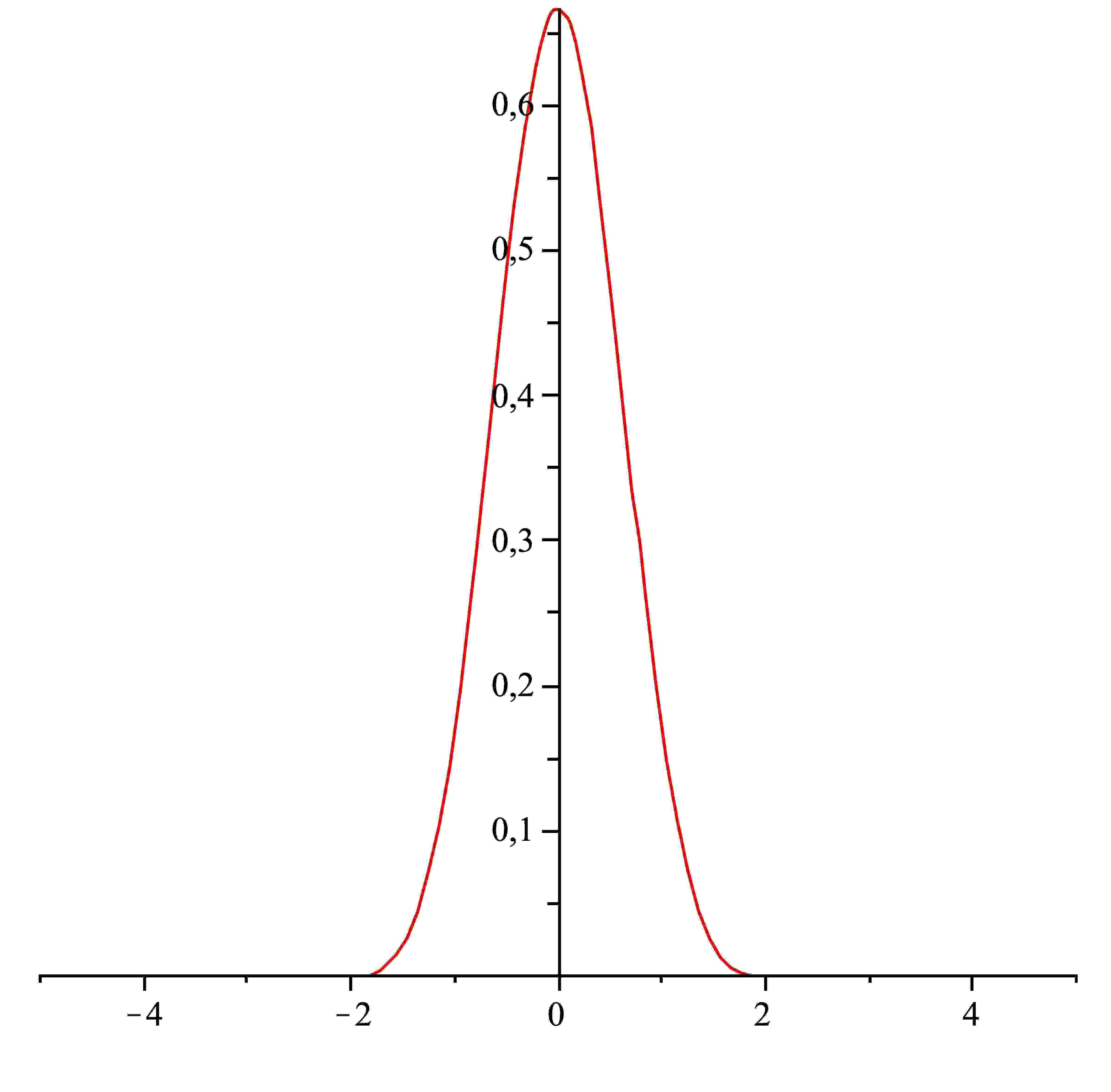}\end{center}\end{figure}
\begin{figure}
\begin{center}\includegraphics[width=.7\textwidth,height=.47\textwidth]
{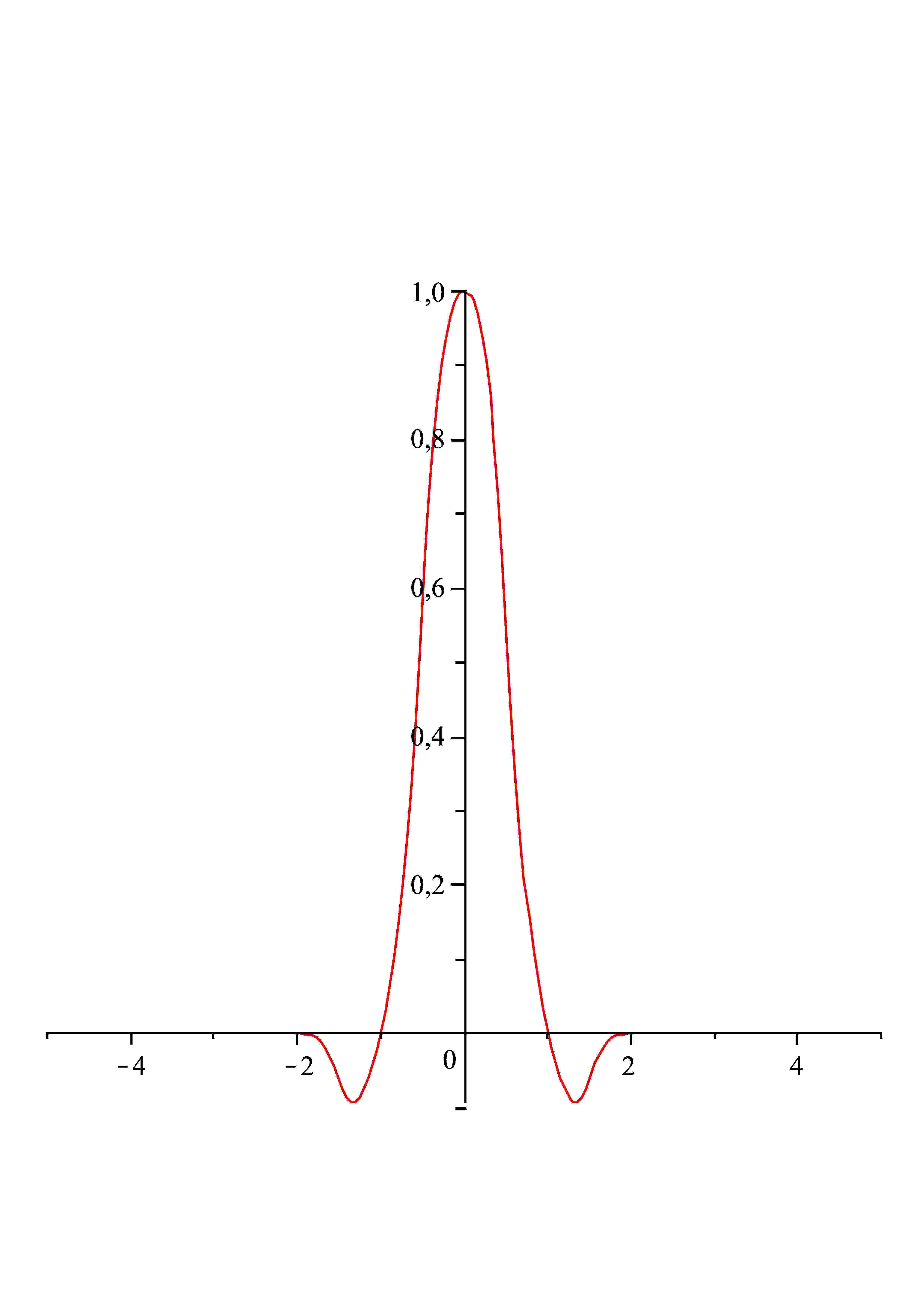}\end{center}
\caption{\small{\it The graphs of $M_3(x)$, $M_4(x)$ and $M(x)$ for $-5\leq x\leq5$.}}
\end{figure}

\vspace{4cm}
 We first consider the \ap Kantorovich-type generalized sampling series''

$$(T^{(2)}_wf)(x)=\sum_{k\in\mathbb Z}M(wx-k)w\int_{k/w}^{(k+1)/w}f(u)du,$$ where

$$f(u)=\begin{cases}40/u^2,&u\mi-5\\-1,&-5\leq u\mi-3\\2,&-3\leq u\mi-2\\-1/2,&-2\leq u\mi-1\\3/2,&-1\leq u\mi0\\1,
&0\leq u\mi1\\-1/2,&1\leq u\mi2  \\-2/u^5,&u\geq2\end{cases}$$

The graphs below represent the approximation of $T^{(2)}_wf(x)$ for $w=10,20,40$.
\begin{figure}[htbp]
\includegraphics[width=.49\textwidth,height=.45\textwidth]
{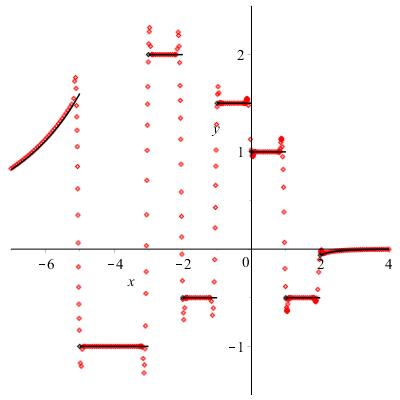}
\includegraphics[width=.49\textwidth,height=.45\textwidth]
{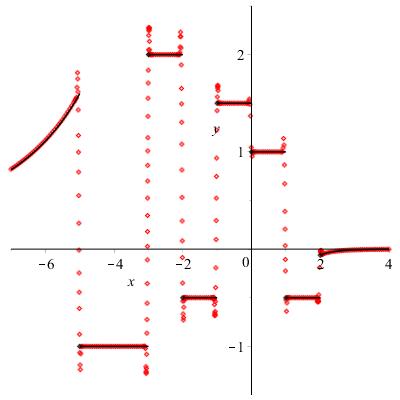}\end{figure}
\begin{figure}[htbp]
\includegraphics[width=.49\textwidth,height=.45\textwidth]
{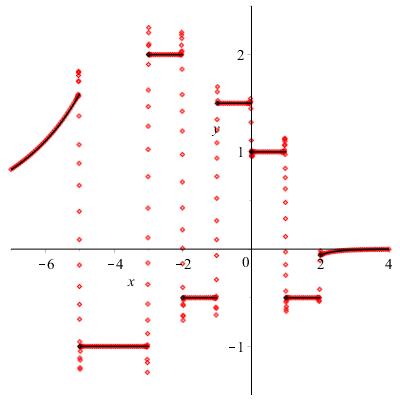}
\caption{\small{\it The graphs of the  functions $T_{5}^{(2)}f(x),T_{15}^{(2)}f(x)$, $T_{40}^{(2)}f(x)$ (red) compared to the graph of $f(x)$ (blue)}}
\end{figure}

\vspace{5cm}
In the next example we consider the Durrmeyer generalized sampling series $T^{(3)}_nf(x)$, in the form of  Section 5. For computational convenience, we take $\psi(u)=F(u)$,  and define $$T_n^{(3)}f(x)=\sum_{k\in\mathbb Z}M(nx-k)\cdot n\int_\mathbb RF(nu-k)f(u)du,$$ where $$f(x)=\begin{cases}\dfrac{1}{u^2},&x<-1\\ -1,&-1\leq x<0 \\	2,&0\leq x<2 \\-\dfrac{3}{u^3},&x\geq 2.\end{cases}$$

The graphs below show the approximation of $T_n^{(3)}f(x)$ as $n=5,10,20$.

\begin{figure}[htbp]
\includegraphics[width=.49\textwidth,height=.45\textwidth]
{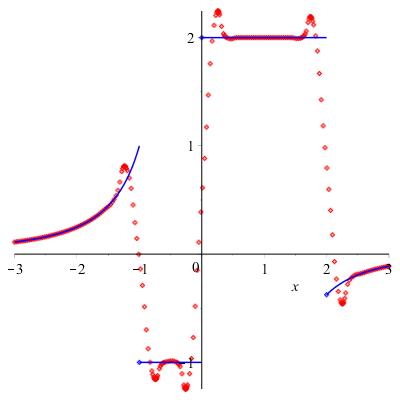}
\includegraphics[width=.49\textwidth,height=.45\textwidth]
{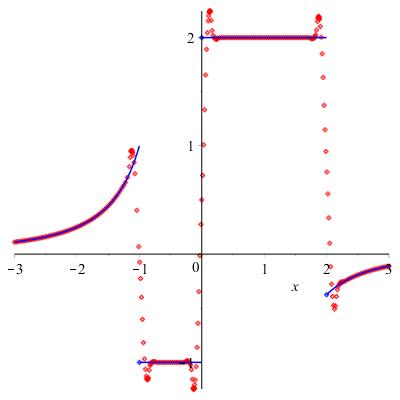}
\includegraphics[width=.49\textwidth,height=.45\textwidth]
{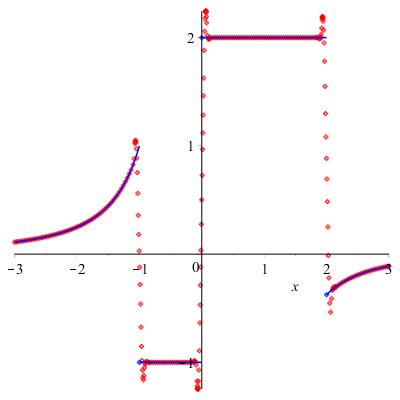}
\vspace{-.2cm}
\caption{\small{\it The graphs of the  functions $T_5^{(3)}f(x),T_{10}^{(3)}f(x)$, $T_{20}^{(3)}f(x)$ (red) compared to the graph of $f(x)$ (blue)}}
\end{figure}
\vspace{2cm}
The last example takes  into account the operator $T^{(7)}_wf$ introduced in Section 3.  In this case, however, we cannot take kernels based on the function $M(u)$ as before, because of the base space $\mathbb R^+$ and of the measure $d\mu(t)=\dfrac{dt}{t}$. Suitable kernel functions in this case are given by

$$\mathcal M_w(u)=\begin{cases} wu^w,&0<u<1,\\0,& \mbox{otherwise.}\end{cases}$$ 

Next, we consider the operators

$$T_w^{(7)}f(x)=\int_0^\infty\mathcal M_w\left(\dfrac{x}{t}\right)\dfrac{1}{2\ln(1+1/w)}\left(\int_{t\frac{w}{w+1}}^{t\frac{w+1}{w}}f(u)\dfrac{du}{u}\right)\dfrac{dt}{t},$$ 

where $$f(x)=\begin{cases}2x,&0\leq x<2,\\1,&2\leq x<4,\\-25/x^3,&x\geq 4\end{cases}.$$
Below (Figure 4), we  represent the approximation of the functions  $T_w^{(7)}f(x)$ for $w=5,20$ and 30 respectively.
\begin{figure}[htbp]
\includegraphics[width=.49\textwidth,height=.45\textwidth]
{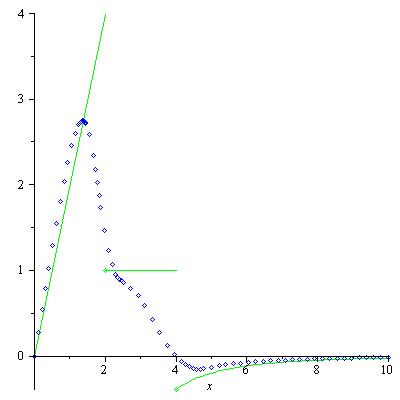}
\includegraphics[width=.49\textwidth,height=.45\textwidth]
{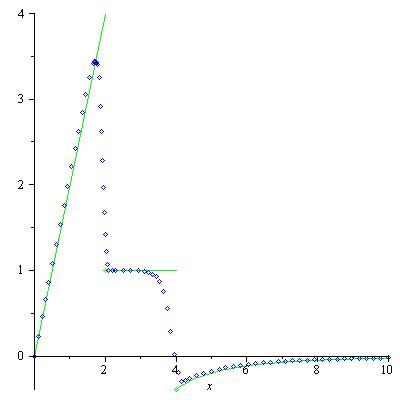}
\includegraphics[width=.49\textwidth,height=.45\textwidth]
{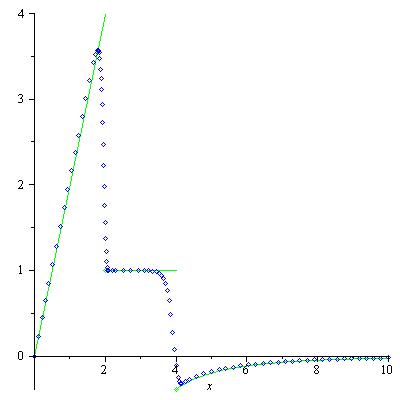}
\caption{\small{\it The graphs of the  functions $T_5^{(7)}f(x),T_{20}^{(7)}f(x)$, $T_{30}^{(7)}f(x)$ (red) compared to the graph of $f(x)$ (green)}}
\end{figure}

\vskip4cm
{\bf Some Concluding Remarks:}
\begin{itemize}
\item[1)] Note that, apart from  the case of operators $T_w^{(3)}f$ (Durrmeyer generalized sampling series), where we  use a uniform convergence result, for the other cases the same approximation  results still hold using the pointwise convergence. And indeed  Theorem \ref{t1}  can be established by similar reasonings  for continuous functions with respect to the pointwise convergence.

\item[2)] We point out that when  we deal with the discontinuous function $f$, the correct way to interpret the approximation results in the above graphical examples is, e.g., to consider approximation with respect to  the $L^p$-norm (i.e., $\varphi(u)=u^p$). On the other hand, when we use continuous functions, the above graphs can be analogously plotted to show  pointwise or uniform approximation of the considered operators to the function $f.$ 

\item[3)]  The theory here introduced  and developed, represents a unified approach in order to study the convergence of  several classes of operators, among them there are integral and discrete operators; in particular  we cover the cases of  the generalized sampling series (with its ''time-jitter'' versions), the sampling Kantorovich ones, the Durrmeyer generalized sampling series  and the cases of convolution  operators (classical, in the Mellin sense, and also of Kantorovich type).

\item[4)]  The previous theory,  set  in a general Orlicz space with a general  $\varphi$-function, represents also a unifying approach  to formulate the approximation  results in several particular cases of Orlicz spaces, interesting by themselves. Among them, for example, there are, as already mentioned in Section 2,  the  $L^p$-spaces ($p \geq 1$),  the  exponential spaces, the interpolation spaces'' or \ap Zygmund spaces''  and many others. The last are very important in the interpolation theory and in the  theory of PDEs.   
 \end{itemize}

\vskip1cm

\noindent{\bf Acknowledgements.}
The authors are members of the Gruppo  
Nazionale per l'Analisi Matematica, la Probabilit\'a e le loro  
Applicazioni (GNAMPA) of the Istituto Nazionale di Alta Matematica (INdAM). 
\noindent The authors are partially supported by the "Department of Mathematics and Computer Science" of the University of Perugia (Italy). 
\rm

\end{document}